\documentclass[a4paper]{article}
\usepackage[utf8]{inputenc}

\usepackage[margin=1in]{geometry}
\usepackage{tikz}

\usepackage{amsmath, amssymb}
\usepackage[hidelinks]{hyperref}

\usepackage{tikz}

\usepackage{xcolor}
\usepackage[dvipsnames]{xcolor}
\definecolor{blauw}{RGB}{24,65,131}
\definecolor{oranje}{RGB}{232,78,15}

\usepackage{amsthm}
\newtheorem{thm}{Theorem}[section]
\newtheorem{lm}[thm]{Lemma}

\newtheorem{crl}[thm]{Corollary}
\newtheorem{prop}[thm]{Proposition}

\newtheorem{conj}[thm]{Conjecture}

\theoremstyle{definition}

\newtheorem{rmk}[thm]{Remark}
\newtheorem{df}[thm]{Definition}

\renewcommand{\phi}{\varphi}

\newcommand{\FF}{\mathbb F}
	\newcommand{\fq}{\FF_q}
	\newcommand{\fp}{\FF_p}

\renewcommand{\leq}{\leqslant}
\renewcommand{\geq}{\geqslant}

\newcommand{\vspan}[1]{\left \langle #1 \right \rangle}

\newcommand{\sett}[2]{ \left\{ #1 \, \, || \, \, #2 \right \} }

\newcommand{\zero}{\mathbf 0}

\newcommand{\mc}{\mathcal C}

\newcommand{\mk}{\mathcal K}

\renewcommand{\mp}{\mathcal P}

\newcommand{\mv}{\mathcal V}

 \DeclareMathOperator{\supp}{supp}
 \DeclareMathOperator{\wt}{wt}

 \DeclareMathOperator{\PG}{PG}
    \newcommand{\pg}{\PG}

 \DeclareMathOperator{\Cyl}{Cyl}

\title{\scshape Even Sets and Dual Projective Geometric Codes: A Tale of Cylinders}
\author{Sam Adriaensen\thanks{Department of Mathematics and Data Science, Vrije Universiteit Brussel, Pleinlaan 2, 1050 Elsene. \href{mailto:sam.adriaensen@vub.be}{sam.adriaensen@vub.be}. The author is supported by a postdoctoral fellowship 12A3Y25N from the Research Foundation Flanders (FWO).}}
\date{}

\begin{document}

\maketitle

\begin{abstract}
 In this paper, we prove that the smallest even sets in ${\rm PG}(n,q)$, i.e.\ sets that intersect every line in an even number of points, are cylinders with a hyperoval as base.
 This fits into a more general study of dual projective geometric codes.
 Let $q$ be a prime power, and define $\mathcal C_k(n,q)^\perp$ as the kernel of the $k$-space vs.\ point incidence matrix of ${\rm PG}(n,q)$, seen as a matrix over the prime order subfield of $\mathbb F_q$.
 Determining the minimum weight of this linear code is still an open problem in general, but has been reduced to the case $k=1$.
 There is a known construction that constructs small weight codewords of $\mathcal C_1(n,q)^\perp$ from minimum weight codewords of $\mathcal C_1(2,q)^\perp$.
 We call such codewords \emph{cylinder codewords}.
 We pose the conjecture that all minimum weight codewords of $\mathcal C_1(n,q)^\perp$ are cylinder codewords.
 This conjecture is known to be true if $q$ is prime.
 We take three steps towards proving that the conjecture is true in general:
 \begin{enumerate}
  \item We prove that the conjecture is true if $q$ is even.
  This is equivalent to our classification of the smallest even sets.
  \item We prove that the minimum weight of $\mc_1(n,q)^\perp$ is $q^{n-2}$ times the minimum weight of $\mc_1(2,q)^\perp$, which matches the weight of cylinder codewords.
  Thus, we completely reduce the problem of determining the minimum weight of $\mc_1(n,q)^\perp$ to the case $n=2$.
  \item We prove that if the conjecture is true for $n=3$, it is true in general.
 \end{enumerate}
\end{abstract}



\section{Introduction}

\subsection{Projective geometric codes}

Throughout this paper, $p$ denotes a prime number, $q$ denotes a power of $p$, and $\fq$ denotes the finite field of order $q$.

For every vector $v = (v_1, \dots, v_n) \in \fq^n$, we define its \emph{support} as the set
\[
 \supp(v) = \sett{i \in \{1,\dots,n\}}{v_i \neq 0},
\]
and its \emph{(Hamming) weight} $\wt(v)$ as the size of its support.
A $k$-dimensional subspace $C$ of $\fq^n$ is called a \emph{linear code}.
The \emph{(Hamming) distance} between two vectors $v,w \in \fq^n$ is defined as
\[
 d(v,w) = \wt(v-w)
\]
which equals the number of coordinates where $v$ and $w$ differ.
The \emph{minimum distance} of a linear code $C$ is defined as
\[
 d(C) = \min \sett{d(v,w)}{v, w \in C,\, v \neq w},
\]
i.e.\ the minimum distance between distinct elements of $C$.
Since $C$ is an additive group, it is easy to see that
\[
 d(C) = \min \sett{\wt(v)}{v \in C \setminus \{\zero\}},
\]
hence we shall also call $d(C)$ the \emph{minimum weight} of $C$.

The $n$-dimensional projective space over $\fq$, consisting of the subspaces of $\fq^{n+1}$, will be denoted by $\pg(n,q)$.
When we refer to the dimension of spaces, we refer to their projective dimension, which is one less than their vector space dimension.
We call subspaces of dimension 0, 1, 2, and $n-1$ in $\pg(n,q)$ respectively \emph{points}, \emph{lines}, \emph{planes}, and \emph{hyperplanes}.
A subspace of dimension $k$ is simply called a \emph{$k$-space}.
We identify each subspace with the set of points incident to it.
If $S$ is a set of points, and $\pi$ is a subspace with $|S \cap \pi| = m$, we say that $\pi$ is \emph{$m$-secant} to $S$.
We use $\langle \pi, \rho \rangle$ to denote the span of two subspaces $\pi$ and $\rho$.

Let $\mp$ be the set of points of $\pg(n,q)$ and $\mk$ the set of its $k$-spaces.
We make a $p$-ary matrix $M$ where the rows correspond to the elements of $\mk$ and the columns to the elements of $\mp$.
Formally, $M \in \fp^{\mk \times \mp}$.
We define the matrix by
\[
 M(\kappa,P) = \begin{cases}
  1 & \text{if } P \in \kappa, \\
  0 & \text{otherwise}.
 \end{cases}
\]
The rowspace of $M$ is denoted as $\mc_{k}(n,q)$, and its kernel as $\mc_{k}(n,q)^\perp$.
We call $\mc_k(n,q)$ the \emph{code of points and $k$-spaces in $\pg(n,q)$} and $\mc_k(n,q)^\perp$ its \emph{dual code}.
Since each codeword $c$ in $\mc_k(n,q)$ or $\mc_k(n,q)^\perp$ is a vector over $\fp$ with coordinates labelled by $\mp$, we will identify $c$ with a function $\mp \to \fp$.

Much is known about the minimum weight of the codes $\mc_k(n,q)$, see e.g.\ \cite{A3-2023-ANoteOn, AdriaensenDenaux2023, assmuskey1992, bagchiinamdar, delsartegoethalsmacwilliams, LavrauwStormeVandeVoorde2010, polverinozullo, SzonyiWeiner}.
We will not discuss it here.
Instead, we will focus on the minimum weight of $\mc_k(n,q)^\perp$.

\subsection{Bounds on the minimum weight of the dual code}

First of all, the problem of determining $d(\mc_k(n,q)^\perp)$ and characterising its minimum weight codewords has been completely reduced to the case $k=1$ by Lavrauw, Storme, and Van de Voorde.

\begin{thm}[{\cite[Theorems 10, 12]{lavrauwstormevandevoorde2008kspaces}}]
 \label{Thm:LSVdV}
 If $k > 1$, then 
 \[
  d(\mc_k(n,q)^\perp) = d(\mc_1(n-k+1,q)^\perp).
 \] 
 Moreover, the minimum weight codewords of $\mc_k(n,q)^\perp$ arise exactly as follows:
 Take an $(n-k+1)$-space $\pi$, a codeword $c$ in the dual code of points and lines in $\pi$, and extend $c$ to a codeword of $\mc_k(n,q)^\perp$ by putting $c(P) = 0$ for all points $P \notin \pi$.
\end{thm}

Minimum weight codewords have been characterised if $q=p$ is prime by Bagchi and Inamdar, and Lavrauw, Storme and Van de Voorde.

\begin{thm}[{\cite[Proposition 2]{bagchiinamdar} \cite[Theorem 12]{lavrauwstormevandevoorde2008kspaces}}]
 \label{Thm:prime}
 If $q=p$ is prime, then $d(\mc_1(n,q)^\perp) = 2q^{n-1}$.
 The minimum codewords are exactly the codewords of the following form for some non-zero scalar $\alpha \in \fp$, and some distinct hyperplanes $\Pi_1$ and $\Pi_2$:
 \[
  c:\mp\to\fp: P \mapsto \begin{cases}
   \alpha & \text{if } P \in \Pi_1 \setminus \Pi_2, \\
   - \alpha & \text{if } P \in \Pi_2 \setminus \Pi_1, \\
   0 & \text{otherwise}.
  \end{cases}
 \]
\end{thm}

If $q$ is even, the minimum weight of $\mc_1(n,q)^\perp$ has been also been determined exactly.

\begin{thm}[{\cite[Theorem 1]{calkinkeyderesmini}}]
\label{Thm:CKdR}
 If $q$ is even, then
 \[
  d(\mc_1(n,q)^\perp) = q^{n-2}(q+2).
 \]
\end{thm}

The lower bound $d(\mc_1(n,q)^\perp) \geq q^{n-2}(q+2)$ builds on the work of Delsarte \cite{delsarte70bch}, and uses an advanced result of Delsarte that represents the codes $\mc_k(n,q)$ as so-called subfield subcodes of generalised Reed-Muller codes.
Moreover, if $q \leq 8$ is even, the minimum weight codewords of $\mc_1(n,q)^\perp$ have been completely classified \cite[\S 3]{A3-2023-ANoteOn}.

\bigskip

In general, if $q$ is not prime or even, the minimum weight $\mc_1(n,q)^\perp$ has not yet been determined (except for some small values, see Theorem \ref{Thm:p^2}).
Bagchi and Inamdar gave a good lower bound on the minimum weight $\mc_1(n,q)^\perp$.

\begin{thm}[{\cite[Theorem 3]{bagchiinamdar}}]
 \label{Thm:BI}
\[
 d(\mc_1(n,q)^\perp) \geq 2 \left( \frac{q^n-1}{q-1}\left( 1 - \frac 1p \right) + \frac 1p \right).
\]
\end{thm}

In particular for $n=2$, this bound equals $2\left( q - \frac qp + 1 \right)$.
A short elementary proof of the Bagchi-Inamdar bound was given by the author in \cite[Result 3.2]{A3-2023-ANoteOn}.
The bound was recently improved by Csajb\'ok, Longobardi, Marino, and Trombetti if $n > 2$, using the polynomial method.

\begin{thm}[{\cite{csajbok2025lower}}]
 \label{Thm:Csajbok}
 If $q$ is odd, then
\[
 d(\mc_1(n,q)^\perp) \geq 2 \left( q - \frac qp + 1 \right) q^{n-2}.
\]
\end{thm}

For comparison, we also state the best known upper bound on the minimum weight of $\mc_1(n,q)^\perp$.

\begin{thm}[{\cite[Corollary 4.15]{LavrauwStormeVandeVoorde2010}}]
 \label{Thm:UB}
\[
 d(\mc_1(n,q)^\perp) \leq \left( 2q - \frac{q-p}{p-1} \right) q^{n-2}.
\]
\end{thm}

If $q=p^2$, some improved bounds on the minimum weight of $\mc_1(2,q)^\perp$ are known.

\begin{thm}
 \label{Thm:p^2}
\begin{enumerate}
 \item \textnormal{\cite[Theorem 1]{KeydeResmini}} $d(\mc_1(2,9)^\perp) = 15$.
 \item \textnormal{\cite[Corollary 1.2]{ClarkHatfieldKeyWard}} $d(\mc_1(2,25)^\perp) = 45$.
 \item \textnormal{\cite[Corollary 1.5]{deboeckvandevoorde2022}} If $p \geq 7$ is prime, then $d(\mc_1(2,p^2)^\perp) \geq 2p^2-2p+5$.
\end{enumerate}
\end{thm}

\subsection{The cylinder construction and a conjecture}

Next, we discuss an important connection between small weight codewords of $\mc_1(2,q)^\perp$ and small weight codewords of $\mc_1(n,q)^\perp$.
We start by discussing the concept of a cylinder as a point set in $\pg(n,q)$.

\begin{df}
 Consider two disjoint subspaces $\pi$ and $\rho$ of $\pg(n,q)$, and let $S$ be a set of points in $\pi$.
 The \emph{cylinder} with vertex $\rho$ and base $S$ is defined as the set of points
 \[
  \bigcup_{P \in S} \vspan{P, \rho} \setminus \rho.
 \]
\end{df}

This inspires the following connection between small weight codewords of $\mc_1(2,q)^\perp$ and small weight codewords of $\mc_1(n,q)^\perp$.
The construction can be obtained by inductively applying a construction of Bagchi and Inamdar \cite[Lemma 6]{bagchiinamdar}.

\begin{prop}[{\cite[Lemma 6]{bagchiinamdar}}]
 \label{Prop:Cylinder}
 Consider a codeword $c \in \mc_1(2,q)^\perp$.
 Embed $\pg(2,q)$ as a plane $\pi$ in $\pg(n,q)$.
 Choose an $(n-3)$-space $\rho$ disjoint to $\pi$.
 Construct the function $\Cyl_\rho(c)$ on the points $\mp$ of $\pg(2,q)$ by defining
 \[
  \Cyl_\rho(c): \mp \to \fp: P \mapsto \begin{cases}
    c(\vspan{\rho,P} \cap \pi) & \text{if } P \notin \rho, \\
    0 & \text{if } P \in \rho.
  \end{cases}
 \]
 Then $\Cyl_\rho(c)$ is a codeword of $\mc_1(n,q)^\perp$ of weight $q^{n-2} \wt(c)$.
 We call $\Cyl_\rho(c)$ a \emph{cylinder codeword} with vertex $\rho$ and base $c$.
\end{prop}

We pose the following conjecture.

\begin{conj}
 \label{Conj}
 All minimum weight codewords of $\mc_1(n,q)^\perp$ are cylinder codewords, for every prime power $q$, and every integer $n > 2$.
\end{conj}

In this paper, we make three important steps towards resolving the conjecture.
These steps are discussed in detail in the following subsection, but we summarize them here.
\begin{enumerate}
 \item We show that the minimum weight of $\mc_1(n,q)^\perp$ matches the weight of cylinder codewords.
 In combination with Theorem \ref{Thm:LSVdV}, this completely reduces the determination of $d(\mc_k(n,q)^\perp)$ to determining $d(\mc_1(2,q)^\perp)$.
 \item We prove that if Conjecture \ref{Conj} is true for $n=3$, then it is true for all $n$.
 \item We show that Conjecture \ref{Conj} is true if $q$ is even.
\end{enumerate}

Note that Theorem \ref{Thm:prime} shows that Conjecture \ref{Conj} is true if $q$ is prime.

\subsection{New results}

\subsubsection{Cylinder codewords have minimum weight}

Note that if $c$ is a minimum weight codeword of $\mc_1(2,q)^\perp$, and we apply the cylinder construction in $\pg(n,q)$, then $\Cyl_\rho(c)$ has weight $q^{n-2} d(\mc_1(2,q)^\perp)$.
We show that this is the true minimum weight of $\mc_1(n,q)^\perp$.

\begin{thm}
 \label{Thm:Main:Wt}
 If $n > 2$, then
\[
 d(\mc_1(n,q)^\perp) = q^{n-2}d(\mc_1(2,q)^\perp).
\]
\end{thm}

Using Theorem \ref{Thm:LSVdV}, we immediately obtain the following corollary.

\begin{crl}
 \label{Crl:Main:1}
If $n>2$, then
\[
 d(\mc_k(n,q)^\perp) = q^{n-k-1}d(\mc_1(2,q)^\perp)
\]
\end{crl}

Note in particular that Theorem \ref{Thm:Main:Wt} in combination with the Bagchi-Inamdar (Theorem \ref{Thm:BI}) bound for $n=2$ recovers Theorem \ref{Thm:Csajbok}.
Moreover, since it is a well-known fact that if $q$ is even, then $d(\mc_1(2,q)^\perp) = q+2$ (as we will see later), Theorem \ref{Thm:Main:Wt} also reproves Theorem \ref{Thm:CKdR}.
Both Theorem \ref{Thm:BI} and Theorem \ref{Thm:Main:Wt} can be proven with elementary tools that only rely on the combinatorics of projective spaces, not the underlying algebra.
Hence, we obtain alternative proofs of Theorem \ref{Thm:Csajbok} and Theorem \ref{Thm:CKdR} that avoid the use of polynomials and advanced methods.

Moreover, Theorem \ref{Thm:p^2} and Corollary \ref{Crl:Main:1} imply the following.

\begin{crl}
 \begin{enumerate}
  \item $d(\mc_k(n,9)^\perp) = 9^{n-k-1} \cdot 15$
  \item $d(\mc_k(n,25)^\perp) = 25^{n-k-1} \cdot 45$
  \item If $p \geq 7$ is prime, then $d(\mc_k(n,p^2)^\perp) \geq p^{2(n-k-1)} (2p^2-2p+5)$.
 \end{enumerate}
\end{crl}

\subsubsection{Reduction to three dimensions}

We will reduce the resolution of Conjecture \ref{Conj} to the 3-dimensional case.

\begin{thm}
 \label{Thm:Main:3 dim}
 Let $q$ be a prime power.
 Suppose that all minimum weight codewords of $\mc_1(3,q)^\perp$ are cylinder codewords.
 Then for any $n \geq 3$, all minimum weight codewords of $\mc_1(n,q)^\perp$ are cylinder codewords.
\end{thm}

We note that Conjecture \ref{Conj} holds if $q$ is prime by Theorem \ref{Thm:prime}, so the above theorem trivially holds for prime values of $q$.
Moreover, it was already proven to hold for even values of $q$ by the author \cite[Proposition 3.9]{A3-2023-ANoteOn}.

\subsubsection{Even prime powers and even sets}

We extend Theorem \ref{Thm:CKdR} by characterising all minimum weight codewords of $\mc_1(n,q)^\perp$ in case $q$ is even.

\begin{thm}
 \label{Thm:Main:Even:Code}
 If $q$ is even, the minimum weight of $\mc_1(n,q)^\perp$ is $q^{n-2}(q+2)$, and the only minimum weight codewords are cylinder codewords.
\end{thm}

If $q$ is even, then any codeword $c$ of $\mc_k(n,q)^\perp$ is a function $\mp \to \FF_2 = \{0,1\}$, and we can identify $c$ with its support.
This brings us to the following definition.

\begin{df}
 We call a non-empty set $S$ of points in $\pg(n,q)$ an \emph{even set}\footnote{sometimes called a \emph{set of even type}} with respect to $k$-spaces if $S$ intersects every $k$-space in an even number of points.
 If $S$ is an even set with respect to lines, we simply call $S$ an even set.
\end{df}

Note that if $\mp$ is the set of points of $\pg(n,q)$, then a non-zero function $c:\mp \to \FF_2$ is a codeword of $\mc_k(n,q)^\perp$ if and only if $\supp(c)$ is an even set with respect to $k$-spaces.

If $S$ is an even set in $\pg(2,q)$, and $P \in S$, then each of the $q$ lines through $P$ contain at least one other point of $S$, hence $|S| \geq q+2$.

\begin{df}
An even set in $\pg(2,q)$ of size $q+2$ is called a \emph{hyperoval}.
It is known that whenever $q$ is even, $\pg(2,q)$ contains hyperovals, see e.g.\ \cite[Theorem 6.9]{KissSzonyi}.
\end{df}

Therefore, \ref{Thm:Main:Even:Code} yields the following equivalent formulation in the language of even sets.

\begin{thm}
 \label{Thm:Main:Even:Set}
 If $q$ is even, then the smallest even sets in $\pg(n,q)$ have size $q^{n-2}(q+2)$, and they are exactly the cylinders with an $(n-3)$-dimensional vertex and hyperoval as base.
\end{thm}

For completion, we can use Theorem \ref{Thm:LSVdV} to extend the previous theorem to even sets with respect to any dimension.

\begin{crl}
 \label{Crl:Main:2}
 If $q$ is even, then the smallest even sets in $\pg(n,q)$ with respect to $k$-spaces have size $q^{n-k-1}(q+2)$, and are exactly the cylinders with an $(n-k-2)$-dimensional vertex and a hyperoval as base.
\end{crl}

\begin{rmk}
 For any even prime power $q$, Corollary \ref{Crl:Main:2} classifies the smallest even sets with respect to $k$-spaces in projective spaces over $\fq$, in as far as the hyperovals of $\pg(2,q)$ are classified.
 This has currently been done for $q \leq 64$ \cite{Vandendriessche}, but is still open for larger values of $q$.
\end{rmk}

\subsection{Organisation of the paper and proof techniques}

The remainder of the paper is only concerned with proving Theorems \ref{Thm:Main:Wt}, \ref{Thm:Main:3 dim}, and \ref{Thm:Main:Even:Set}.
This will be done in Sections \ref{Sec:Main:Wt}, \ref{Sec:All?}, and \ref{Sec:Main:Even}, respectively.

Let us make a note on the proof techniques.
Several key steps use a variance-type argument.
One can think of this as follows: if $S$ is a set of points in $\pg(n,q)$, then the variance of the distribution of $S$ over the hyperplanes\footnote{or any more generally over the $k$-spaces for some $k \geq 1$} of $\pg(n,q)$ is upper bounded.
In other words, it is not possible that there are  many hyperplanes $\pi$ for which $|\pi \cap S|$ deviates much from $|S|/q$.

To prove Theorem \ref{Thm:Main:Wt}, we show that if $d(\mc_1(n,q)^\perp) < q^{n-2} d(\mc_1(2,q)^\perp)$, there are too many hyperplanes that don't intersect the support of a given minimum weight codeword, which yields a contradiction.
To prove Theorem \ref{Thm:Main:Even:Set}, we take an even set $S$ of size $q(q+2)$ in $\pg(3,q)$, and study the smallest number $m>2$ such that there exists an $m$-secant line to $S$.
In case $m < q$, we force there to be too many hyperplanes that either don't intersect $S$, or intersect $S$ in significantly more than $q+2$ points, which yields a contradiction.

\section{Proof of Theorem \texorpdfstring{\ref{Thm:Main:Wt}}{}}
 \label{Sec:Main:Wt}

We will use the notation
\[
 \theta_n = \frac{q^{n+1}-1}{q-1} = q^n + q^{n-1} + \dots + q + 1
\]
for the number of points in $\pg(n,q)$.
In order to prove Theorem \ref{Thm:Main:Wt}, we need two tools.

First of all we need a bound on so-called incidence-free sets in $\pg(n,q)$.
We apply the bound from Haemers \cite{haemers95}, which holds in general 2-designs, to the specific case of points and hyperplanes in $\pg(n,q)$.
This theorem can be proven using linear algebra, but also through an easy combinatorial method, cf.\ \cite[\S 2]{dewinterschillewaertverstraete}.

\begin{thm}[{\cite[Corollary 5.3]{haemers95}}]
 \label{Thm:Haemers}
 Suppose that $S$ is a set of points and $T$ is a set of hyperplanes in $\pg(n,q)$ such that no point of $S$ is incident to a hyperplane in $T$.
 Then
 \[
  \theta_{n-1}^2 |S| |T| \leq q^{n-1} (\theta_n - |S|) (\theta_n - |T|).
 \]
\end{thm}

Secondly, we need a simple but crucial observation.

\begin{lm}
 \label{Lm:Restriction}
 Suppose that $c \in \mc_1(n,q)^\perp$ and $\pi$ is a subspace of $\pg(n,q)$ of dimension $m \geq 2$.
 Then the restriction $c|_\pi$ of $c$ to the points of $\pi$ is a codeword of the dual code of points and lines in $\pi$.
 Therefore, $c|_\pi$ is either the zero codeword, or has weight at least $d(\mc_1(m,q)^\perp)$.
\end{lm}

\begin{proof}
 A function $c:\mp \to \fp$ belongs to $\mc_1(n,q)^\perp$ if and only if $\sum_{P \in \ell} c(P) = 0$ for all lines $\ell$ of $\pg(n,q)$.
 This applies in particular to the lines of $\pi$, hence $c|_\pi$ is in the dual code of points and lines in $\pi$.
 The last sentence follows from the definition of minimum weight.
\end{proof}

We are now ready to prove the theorem.

\begin{proof}[Proof of Theorem \ref{Thm:Main:Wt}]
 We prove the theorem by induction on $n$.
 We use $n=2$ as base case, for the theorem statement is trivially true.
 Thus, suppose that $n > 2$, and the theorem holds for $\mc_1(n-1,q)^\perp$.
 Write $d_2 = d(\mc_1(2,q)^\perp)$.
 Let $c \in \mc_1(n,q)^\perp$ be a non-zero codeword of minimum weight.
 Since $\mc_1(n,q)^\perp$ contains cylinder codewords with minimum weight codewords of $\mc_1(2,q)^\perp$ as base, we know that $\wt(c) \leq q^{n-2} d_2$.
 Suppose by way of contradiction that $\wt(c) < q^{n-2} d_2.$
 
 There exists a 0-secant hyperplane $\pi$ to $\supp(c)$.
 Otherwise, by Lemma \ref{Lm:Restriction}, every hyperplane contains at least $d(\mc_1(n-1,q)^\perp) = q^{n-3} d_2$ points of $\supp(c)$.
 Since every point of $\supp(c)$ is incident with $\theta_{n-1}$ hyperplanes, this implies that
 \[
  \wt(c) \geq \frac{\theta_{n}}{\theta_{n-1}} q^{n-3} d_2
  = \frac{q^{n+1}-1}{q^n-1} q^{n-3} d_2
  > q^{n-2} d_2,
 \]
 which contradicts our assumptions.

 Now take an $(n-2)$-space $\rho$ in $\pi$.
 Then $\rho$ is incident to at least one other 0-secant hyperplane to $\supp(c)$.
 Otherwise, all other $q$ hyperplanes through $\rho$ contain at least $q^{n-3} d_2$ points of $\supp(c)$, but then $\wt(c) \geq q^{n-2} d_2$, contrary to our assumptions.
 Since there are $\theta_{n-1}$ $(n-2)$-spaces in $\pi$, we find in total a set $T$ of $\theta_{n-1}+1$ hyperplanes of $\pg(n,q)$ that are 0-secant to $\supp(c)$.

 On the other hand, let $P$ be a point of $\supp(c)$.
 Since $\sum_{Q \in \ell}c(Q) = 0 \neq c(P)$ for every line through $P$, we can choose a point $Q \in \supp(c) \setminus \{ P \}$ on each of the $\theta_{n-1}$ lines through $P$.
 If we add $P$ to this collection of points, we obtain a subset $S$ of $\supp(c)$ of size at $\theta_{n-1}+1$.

 Therefore, $S$ and $T$ are sets of $\theta_{n-1}+1$ points and hyperplanes respectively, with no point of $S$ incident with a hyperplane of $T$.
 By Theorem \ref{Thm:Haemers}, this implies that
 \[
  \theta_{n-1}^2 (\theta_{n-1}+1)^2 \leq q^{n-1} (q^n-1)^2
 \]
 Using that $\theta_{n-1} = \frac{q^n-1}{q-1}$, and taking square roots, this inequality is equivalent to
 \[
  \theta_{n-1} + 1 \leq q^\frac{n-1}2 (q-1).
 \]
 Since $n \geq 3$, we have $\frac{n-1}2 \leq n-2$, hence the inequality implies that
 \[
  q^{n-1} + q^{n-2} + \dots + q + 2 = \theta_{n-1} + 1 \leq q^{n-2}(q-1),
 \]
 which is clearly a contradiction.
\end{proof}

\section{Proof of Theorem \texorpdfstring{\ref{Thm:Main:3 dim}}{}}
 \label{Sec:All?}

In this section, we prove that we can reduce Conjecture \ref{Conj} to the 3-dimensional case.

\begin{lm}
 \label{Lm:Cyl Collineation}
 Consider a cylinder codeword $c = \Cyl_\rho(v)$ of $\mc_1(n,q)^\perp$, with $\rho$ an $(n-3)$-space, and $v$ a codeword of $\mc_1(2,q)^\perp$.
 Then for any plane $\pi$ disjoint to $\rho$, there exists a collineation $\phi: \pi_1 \to \pg(2,q)$ such that $c|_{\pi} = v \circ \phi$.
 In particular, $\wt(c|_\pi) = \wt(v)$.
\end{lm}

\begin{proof}
 In order to construct $c = \Cyl_\rho(v)$, we need to choose a plane $\sigma$ in $\pg(n,q)$, disjoint to $\rho$, such that $v$ is a codeword in the dual code of points and lines in $\sigma$.
 Now take a plane $\pi$ disjoint to $\rho$.
 By Grassmann's identity, every $(n-2)$-space through $\rho$ intersects $\sigma$ and $\pi$ in a point.
 This yields a map $\phi: \pi \to \sigma: P \mapsto \vspan{P,\rho} \cap \sigma$.
 It is easy to see that a line $\ell$ of $\pi$ is mapped to the line $\vspan{\ell,\rho}$ of $\sigma$, hence $\phi$ is a collineation.
 Moreover, $c|_\pi(P) = c(P) = v(\vspan{\rho,P} \cap \sigma) = v(\sigma(P))$ for each point $P \in \pi$, thus $c|_\pi = v \circ \phi$.
\end{proof}

\begin{proof}[Proof of Theorem \ref{Thm:Main:3 dim}]
 If $q$ is prime, then Theorem \ref{Thm:prime} proves that Theorem \ref{Thm:Main:3 dim} is true.
 Hence, for the remainder of the proof, we assume that $q$ is not prime.
 
 We prove the theorem by induction.
 Take an integer $n > 3$, and suppose that for all integers $m$ with $3 \leq m < n$ all minimum weight codewords of $\mc_1(m,q)^\perp$ are cylinder codewords.
 Let $c$ be a minimum weight codeword of $\mc_1(n,q)^\perp$.
 Denote $d(\mc_1(2,q)^\perp)$ by $d_2$, so that $\wt(c) = q^{n-2} d_2$.
 Since we assumed that $q$ is not prime, Theorem \ref{Thm:UB} tells us that $d_2 < 2q$.
 Denote $\supp(c)$ as $S$.

 \bigskip

 \noindent {\it Step 1: There exists a plane $\pi$ such that $c|_\pi$ has weight $d_2$.}
 
 Using the same arguments as in the beginning of the proof of Theorem \ref{Thm:Main:Wt}, we can see that there exists a hyperplane $\Pi$ disjoint to $S$.
 If every $(n-2)$-space in $\Pi$ is incident with at least one other 0-secant hyperplane to $S$, we can derive a contradiction in the same fashion as in the proof of Theorem \ref{Thm:Main:Wt}.
 Thus, there must exist an $(n-2)$-space $\rho$ in $\Pi$ such that the other $q$ hyperplanes through $\rho$ intersect $S$.
 Since they each intersect $S$ in at least $q^{n-3} d_2$ points, and $|S| = q^{n-2} d_2$, they all intersect $S$ in a minimum weight codeword of $\mc_1(n-1,q)^\perp$.
 Take such a hyperplane $\Pi_2$ through $\rho$.
 Then $c|_{\Pi_2}$ is a cylinder codeword by the induction hypothesis, with some vertex $\rho'$, and every plane of $\Pi_2$ disjoint to $\rho'$ intersects $S$ in $d_2$ points by Lemma \ref{Lm:Cyl Collineation}.

 \bigskip

 \noindent {\it Step 2: For every $3$-space $\sigma$ through $\pi$, $c|_\sigma$ is a cylinder codeword.}

 Every $3$-space $\sigma$ through $\pi$ intersects $S$, hence it intersects $S$ in at least $q d_2$ points.
 This means that $\sigma$ contains at least $(q-1)d_2$ points of $S \setminus \pi$.
 Note that there are $\theta_{n-3}$ $3$-spaces through $\pi$.
 If any $3$-space intersects $S$ in more points, then
 \[
  |S| > d_2 + \theta_{n-3} (q-1) d_2 = q^{n-2} d_2,
 \]
 which contradicts that $c$ has minimum weight.
 By the hypothesis of the theorem, $c|_\sigma$ is a cylinder codeword for every $3$-space $\sigma$ through $\pi$.

 \bigskip

 \noindent {\it Step 3: Recognising the vertices of $\pi$.}

 For every $3$-space $\sigma$ through $\pi$, $c|_\sigma$ is a cylinder codeword, hence it has some point $P$ as vertex.
 We call such a point $P$ a \emph{vertex} of $\pi$.
 Let $\mv_\pi$ denote the set of vertices of $\pi$.
 Note that $|\mv_\pi| = \theta_{n-3}$.
 We claim that for any point $Q \in S \cap \pi$,
 \begin{align}
  \label{Eq:Vertex}
  \mv_\pi = \sett{P \notin \pi}{\vspan{P,Q} \setminus S = \{P\}},
 \end{align}
 i.e.\ $P$ is a vertex of $\pi$ if and only if $P$ is not in $S$ or $\pi$, and the line $\vspan{P,Q}$ is $q$-secant to $S$.
 By the definition of cylinder codewords, it immediately follows that $\mv_\pi$ contains all vertices of $\pi$.
 Vice versa, suppose that $P \in \mv_\pi$, let $\sigma = \vspan{P,\pi}$, and suppose to the contrary that $R$ is the vertex of the cylinder codeword $c|_\sigma$.
 Consider the line $\ell = \vspan{P,Q}$.
 Since $R \notin S$ and $\ell \setminus S = \{P\}$, we know that $R \notin \ell$.
 Then there exists a plane $\pi_2$ in $\sigma$ through $\ell = \vspan{P,Q}$ that does not contain $R$.
 Since $c|_\sigma$ is a cylinder codeword, $c|_{\pi_2}$ is a minimum weight codeword of $\mc_1(2,q)^\perp$ by Lemma \ref{Lm:Cyl Collineation}.
 However, every line through $Q$ contains at least one other point of $S$, so $|\pi_2 \cap S| \geq |\ell \cap S| + q = 2q > d_2$, a contradiction.
 Therefore, $\mv_\pi$ is indeed the set of vertices of $\pi$.

 \bigskip

 \noindent {\it Step 4: The vertices of $\pi$ constitute a subspace.}

 Suppose to the contrary that $\mv_\pi$ is not a subspace.
 Then there exist two points $P_1, P_2 \in \mv_\pi$ such that $\ell = \vspan{P_1,P_2}$ contains a point $R \notin \mv_\pi$.
 Note that since $\vspan{P_1,\pi} \neq \vspan{P_2,\pi}$, every point of $\ell$ is in a different $3$-space through $\pi$.
 Let $P_3$ be the vertex of $\sigma_3 = \vspan{R,\pi}$.
 Let $R'$ be the point $\vspan{R,P_3} \cap \pi$, and take a line $\ell_2$ in $\pi$ that does not contain $R'$, and does contain some point $Q \in S$.
 Then the plane $\pi_2 = \vspan{\ell_2,R}$ does not contain $P_3$.
 Since $c|_{\sigma_3}$ is a cylinder codeword with vertex $P_3$, $c|_{\pi_2}$ is a minimum weight codeword of $\mc_1(2,q)^\perp$.
 Therefore, we know from Step 2 that every $3$-space through $\pi_2$ yields a cylinder codeword, and we obtain a set $\mv_{\pi_2}$ of vertices of $\pi_2$.
 Moreover, since $\pi$ and $\pi_2$ share the point $Q \in S$, it follows from (\ref{Eq:Vertex}) that $\mv_\pi = \mv_{\pi_2}$.
 However, this now yields a contradiction, because then $P_1$ and $P_2$ are different vertices of $\pi_2$, but $\vspan{P_1,P_2}$ intersects $\pi_2$, hence $\vspan{P_1,\pi_2} = \vspan{P_2,\pi_2}$.

 \bigskip

 \noindent {\it Step 5: $c$ is a cylinder codeword with vertex $\mv_\pi$ and base $c|_\pi$.}

 Write $\rho$ for the subspace $\mv_\pi$.
 We check that $c = \Cyl_\rho(c|_\pi)$.
 Take a point $R$ in $\pg(n,q)$.
 If $R \in \rho$, then $c(R) = \Cyl_\rho(c|_\pi)(R) = 0$.
 If $R \in \pi$, then $\Cyl_\rho(c|_\pi)(R) = c|_\pi(R) = c(R)$.
 If $R \notin \rho$ and $R \notin \pi$, then let $\sigma$ be $\vspan{R,\pi}$.
 Note that since $\pi$ and $\rho$ are disjoint, $\sigma$ and $\rho$ share at most one point.
 By Grassmann's identity they share at least one point.
 So they share exactly one point $P$.
 Let $R'$ be the point $\vspan{P,R} \cap \pi$.
 Then $\vspan{R,\rho} \cap \pi$ is the point $R'$, so $\Cyl_\rho(c|_\pi)(R) = c|_\pi(R')$.
 On the other hand, we know that $c|_\sigma = \Cyl_P(c|_\pi)$, hence $c(R)=  c|_\sigma(R) = c|_\pi(\vspan{P,R}\cap \pi) = c|_\pi(R')$.
 We may conclude that $c$ indeed equals $\Cyl_\rho(c|_\pi)$.
\end{proof}

\section{Proof of Theorem \texorpdfstring{\ref{Thm:Main:Even:Set}}{}}
 \label{Sec:Main:Even}

In this section, we prove Theorem \ref{Thm:Main:Even:Set}, which is equivalent to Theorem \ref{Thm:Main:Even:Code}.
By Theorem \ref{Thm:Main:3 dim} (or by \cite[Proposition 3.9]{A3-2023-ANoteOn}), it suffices to prove Theorem \ref{Thm:Main:Even:Set} in $\pg(3,q)$.
We will call cylinders with a one-dimensional vertex and a hyperoval base in $\pg(3,q)$ simply \emph{hyperoval cylinders}.
We recall a proposition that helps us recognise hyperoval cylinders.

\begin{prop}[{\cite[Proposition 3.13]{A3-2023-ANoteOn}}]
 \label{Prop:qsecant}
 Suppose that $S$ is an even set in $\pg(3,q)$ of size $q(q+2)$.
 If there exists a $q$-secant line to $S$, then $S$ is a hyperoval cylinder.
\end{prop}

Remark that Proposition \ref{Prop:qsecant} immediately implies that the only even sets of size $q(q+2)$ in $\pg(3,q)$ for $q=2$ and $q=4$ are hyperoval cylinders.
We will now prove that this also holds for $q > 4$.
This recovers the case $q=8$, that was already proven by the author in \cite[Proposition 3.16]{A3-2023-ANoteOn}.

For the remainder of this section, we fix the following notation.
\begin{itemize}
 \item $q \geq 8$ is a power of 2,
 \item $S$ is an even set of size $q(q+2)$,
 \item $m$ is the smallest integer satisfying $m > 2$ such that there exists an $m$-secant line to $S$,
 \item we call every line that intersects $S$ in more that $2$ points a \emph{large secant},
 \item $n_i$ is the number of $i$-secant planes to $S$.
\end{itemize}

Note that the number $m$ is well-defined.
If not, every line would intersect $S$ in either $0$ or $2$ points.
If $P \in S$, then each line through $P$ contains exactly one extra point of $S$ from which it follows that $|S| = q^2+q+2 < q(q+2)$ if $q>2$.

We will need the following definition.

\begin{df}
 An even set $T$ in $\pg(2,q)$ is called a \emph{KM-arc} (named after Korchm\'aros and Mazzocca \cite{KM}) if there exists an even integer $t > 2$ such that $|T| = q+t$, and every line intersects $T$ in 0, 2, or $t$ points.
 We say that $t$ is the \emph{type} of the KM-arc.
\end{df}

We need two basic properties of KM-arcs, and give a proof to be complete.

\begin{lm}
 \label{Lm:KM}
 Suppose that $T$ is a KM-arc of type $t$ in $\pg(2,q)$.
 Then
 \begin{enumerate}
  \item \label{Lm:KM:1} $t$ is a power of $2$,
  \item \label{Lm:KM:4} there are $\frac q2 - \left( \frac t 2 + \frac 1t - 1\right)q$ $0$-secant lines to $T$.
 \end{enumerate}
\end{lm}

\begin{proof}
 Take a point $P \in T$.
 Every line through $P$ is either $2$- or $t$-secant to $T$.
 Since $|T|=q+t$, it readily follows that $P$ is incident with one $t$-secant line, and $q$ $2$-secant lines.
 By double counting, this means that there are $\frac{q+t}t$ $t$-secant lines and $\frac{q(q+t)}2$ $2$-secant lines.
 In particular, since $\frac{q+t}t$ must be integer, $t$ must divide $q$, and hence be a power of $2$.
 This leaves us with
 \[
  (q^2+q+1) - \frac{q+t}t - q \frac{q+t}2 = \frac 12 q^2 - \left( \frac t2 + \frac 1t - 1\right)q
 \]
 $0$-secant lines.
\end{proof}

\begin{lm}
 \label{Lm:M}
 \begin{enumerate}
  \item \label{Lm:M:1} There exists a plane $\pi$ that intersects $S$ in a KM-arc of type $m$.
  \item \label{Lm:M:2} It holds that $m$ is a power of $2$, and if $q > 8$, then $m \neq \frac q2$.
  \item \label{Lm:M:3} There are at least $q^2 - \left( m + \frac 2{m} - 2 \right) q$ 0-secant planes to $S$.
 \end{enumerate}
\end{lm}

\begin{proof}
 (1) Take an $m$-secant line $\ell$ to $S$.
 Take a point $P \in \ell \cap S$.
 Since every line through $P$ contains another point of $S$, every plane through $\ell$ contains at least $q$ more points of $S$.
 If every plane contains more than $q$ extra points of $S$, then $|S| \geq m + (q+1)(q+2) > q(q+2)$, a contradiction.
 Hence there exists a $(q+m)$-secant plane $\pi$ to $S$.
 Since every line of $\pi$ intersects $\pi \cap S$ in at most $2$ or at least $m$ points, it easily follows that $S \cap \pi$ is a KM-arc of $\pi$ of type $m$.

 (2) Since a KM-arc of type $m$ exists, $m$ must be a power of $2$.
 Now suppose that $q > 8$ and $m = \frac q 2$.
 Then $S$ is not a hyperoval cylinder, so by Proposition \ref{Prop:qsecant} there are no $q$-secant lines to $S$.
 Therefore, every point of $S$ is incident with at least $2$ large secant lines.
 It is straightforward to check that every point of $S$ is incident to one $\frac q2$-secant and one $\left(\frac{q}2+2\right)$-secant line.
 Therefore, the $\left( \frac q2 + 2 \right)$-secant lines partition $S$, thus $\frac q2 + 2$ must divide $q(q+2) = q^2 + 2q$.
 Since $\left( \frac q2 + 2 \right)(2q-4) = q^2+2q-8$, we see that $\frac q2 + 2$ must divide 8, but $\frac q2 + 2 > \frac q2 \geq 8$, which yields a contradiction.

 (3) By (1) there exists a plane $\pi$ that intersects $S \cap \pi$ in a KM-arc of type $m$.
 By Lemma \ref{Lm:KM} (\ref{Lm:KM:4}), $\pi$ contains $\frac 12 q^2 - \left( \frac{m}2 + \frac 1{m} - 1 \right)q$ 0-secant lines to $S$.
 Each $0$-secant line $\ell$ in $\pi$ is incident to at least $2$ $0$-secant planes, otherwise there are at least $q-1$ planes through $\ell$, different from $\pi$, that each contain at least $q+2$ points of $S$, hence
 \[
  |S| \geq (q+m) + (q-1)(q+2) = q(q+2) + (m-2) > q(q+2),
 \]
 a contradiction.
 Thus, there are at least $2\left( \frac 12 q^2 - \left( \frac{m}2 + \frac 1{m} - 1 \right)q \right)$ $0$-secant planes to $S$.
\end{proof}

\begin{lm}
 \label{Lm:Standard Eq}
 \begin{align}
  \label{Standard Eq:1}
  \sum_{i>0} (i-(q+2))(i-(q+m)) n_i \leq (q+2) \left( q m^2 - (2q-1)m - q\left(\frac 52 q-1 \right) \right).
 \end{align}
\end{lm}

\begin{proof}
Denote the number of planes that intersect $S$ by $k$.
We know from \ref{Lm:M} (\ref{Lm:M:3}) that $k \leq \theta_3 - q^2 + \left( m + \frac 2m - 2\right)q$.
Then
\begin{align*}
  \sum_{i>0} n_i &= k \leq \theta_3 - q^2 + \left( m + \frac 2m - 2\right)q \\
  \sum_{i>0} i n_i &= |S|\theta_2 = q(q+2)(q^2+q+1) \\
  \sum_{i>0} i^2 n_i &= |S|( \theta_2 + (|S|-1)\theta_1) = q(q+2)q(q^2+4q+2)
 \end{align*}
Indeed, the first equation is by definition of $k$; the second equation follows from double counting pairs $(P,\pi)$ with $P \in S$, and $\pi$ a plane through $P$; the third equaltiy follows from double counting the triples $(P,Q,\pi)$ with $P$ and $Q$ (not necessarily distinct) points of $S$, and $\pi$ a plane through $P$ and $Q$.
Therefore
\begin{align*}
 \sum_{i>0} &(i-(q+2))(i-(q+m)) n_i
 = \sum_i i^2 n_i - (2q+m+2) \sum_i i n_i + (q+2)(q+m) \sum_{i>0} n_i \\
 &= q(q+2)q(q^2+4q+2) - (2q+m+2) q(q+2) (q^2+q+1) + (q+2)(q+m) k \\
 &\leq  (q+2) \left( q^2 (q^2+4q+2) - (2q+m+2)q(q^2+q+1) + (q+m)\left( q^3 + \left( m + \frac 2m -1 \right) q + 1 \right) \right) \\
 &= (q+2) \left( q m^2 - (2q-1) m - q(3q-1) + \frac{2q^2}m \right).
\end{align*}
Since $m \geq 4$, we have $\frac{2q^2}m \leq \frac{q^2}2$, and the lemma follows.
\end{proof}

\begin{lm}
 \label{Lm:LBm}
 $m > \sqrt{2q}$.
\end{lm}

\begin{proof}
 Every plane $\pi$ intersects $S$ in either $0$, $q+2$, or at least $q+m$ points.
 Indeed, if $\pi$ has only $0$- and $2$-secant lines to $S$, then $\pi \cap S$ is either empty or a hyperoval, and otherwise $\pi$ has a line $\ell$ with $|\ell \cap S| \geq m$, whence $|S \cap \pi| \geq |\ell \cap \pi|+q \geq q+m$.
 Therefore, the sum in the left hand side of (\ref{Standard Eq:1}) has only non-negative terms.
 This tells us that
 \[
  q m^2 - (2q-1) m - q \left( \frac52 q-1 \right) \geq 0.
 \]
 The inequality does not hold for $m=0$ and for $m=\sqrt{2q}$, so it does not hold for any intermediary values, since the left hand side is a quadratic polynomial in $m$.
 Therefore, $m > \sqrt{2q}$.
\end{proof}

\begin{lm}
 \label{Lm:Large planes}
 \begin{enumerate}
  \item \label{Lm:Large planes:1} If a plane $\pi$ intersects $S$ in more than $q+m$ points, then it intersects $S$ in at least $q+2m-2$ points.
  \item \label{Lm:Large planes:2} If $m < q$, each $2$-secant line to $S$ is incident with at least one plane that contains at least $q+2m-2$ points of $S$.
  \item \label{Lm:Large planes:3} There are more than $\frac 12 q^3 (q+2)$ $2$-secant lines to $S$.
 \end{enumerate}
\end{lm}

\begin{proof}
 (1)  First consider the case where all points of $\pi \cap S$ are on at most one large secant line in $\pi$.
 Then they must be on exactly one large secant, which intersects $S$ in $t = |\pi \cap S|-q$ points.
 Therefore, $\pi \cap S$ is a KM-arc of type $t > m$ in $\pi$, hence, $t \geq 2m$ and $|S \cap \pi| = q+t \geq q+2m$.
 Now consider the case where some point $P \in \pi \cap S$ is incident with at least $2$ large secants in $\pi$.
 Then both lines are at least $m$-secant, so $|S \cap \pi| \geq 1 + (q-1)\cdot 1 + 2 \cdot(m-1) = q+2m-2$.

 (2) Let $\ell$ be a 2-secant line to $S$ and suppose on the contrary that $\ell$ is not incident with any plane that intersects $S$ in at least $q+2m-2$ points.
 Then every plane through $\ell$ is either $(q+2)$-secant or $(q+m)$-secant to $S$.
 Let $k$ denote the number of $(q+m)$-secant planes through $\ell$.
 Then
 \[
  q(q+2) = |S| = 2 + (q+1-k)q + k(q+m-2) = q(q+1)+2 + k(m-2),
 \]
 hence $k = \frac{q-2}{m-2}$, which implies that $m-2$ divides $q-2$.
 Since $m-2$ divides $(m-2) \frac qm = q - 2 \frac qm$, $m-2$ must divide $(q-2) - (q - 2 \frac qm) = 2 \frac qm - 2$.
 Since $q > m$, this implies that $m \leq 2 \frac qm$, which contradicts our previous result that $m > \sqrt{2q}$.
 This contradiction finishes the proof.

 (3) Take a point $P \in S$.
 Let $k$ denote the number of large secants through $P$.
 Then 
 \[
  q(q+2) = |S| \geq 1 + (q^2+q+1-k) \cdot 1 + k\cdot (m-1) = q^2+q+2 + k (m-2),
 \]
 which implies that $k \leq \frac{q-2}{m-2} < \frac q 2$.
 Therefore, $P$ is on more than $q^2$ 2-secant lines.
 A double count then tells us that there are more than $\frac 12 q(q+2)\cdot q^2$ 2-secant lines to $S$.
\end{proof}

We are now ready to finish the proof.

\begin{proof}[Proof of Theorem \ref{Thm:Main:Even:Set}]
 It suffices to prove that $m = q$.
 If $q = 8$, then $m > \sqrt{2 \cdot 8} = 4$ by Lemma \ref{Lm:LBm}, hence, since $m$ is a power of $2$, $m = 8$.
 Hence, we may suppose that $q > 8$.
 
 Suppose to the contrary that $m < q$, then $q \leq \frac q4$ by Lemma \ref{Lm:M} (\ref{Lm:M:2}).
 We will give a lower bound on the left hand side of (\ref{Standard Eq:1}).
 By Lemma \ref{Lm:Large planes} there are more than $\frac 12 q^3 (q+2)$ $2$-secant lines to $S$, each of which is incident with at least one plane that intersects $S$ in at leas $q+2m-2$ points.
 If $\pi$ is an $i$-secant plane to $S$ and $i > q+2$, then every point of $\pi \cap S$ is on at least one large secant line to $S$ in $\pi$, hence at most $q$ $2$-secant lines, which implies that $\pi$ contains at most $i \frac q2$ $2$-secant lines to $S$.
 Therefore,
 \[
  \sum_{i \geq q+2m-2} i n_i \frac q2 > \frac 12 q^3 (q+2),
 \]
 which implies that
 \[
  \sum_{i \geq q+2m-2} i n_i > q^2 (q+2).
 \]
 Now
 \[
  \sum_{i\geq q+2m-2} (i-(q+2))(i-(q+m)) n_i = \sum_{i \geq q+2m-2} \frac{(i-(q+2))(i-(q+m))}i i n_i.
 \]
 It is straightforward to check, e.g.\ by considering the derivative, that for $i > q+m$, the expression $\frac{(i-(q+2))(i-(q+m))}i$ is increasing in $i$.
 Therefore
 \begin{align*}
  \sum_{i \geq q+2m-2} & \frac{(i-(q+2))(i-(q+m))}i i n_i \\
  & \geq \sum_{i \geq q+2m-2} \frac{((q+2m-2)-(q+2))((q+2m-2)-(q+m))}{q+2m-2} i n_i \\
  & > \frac{(2m-4)(m-2)}{q+2m-2} q^2(q+2).
 \end{align*}
 Now using that $m \leq \frac q4$, we have $q+2m-2 < \frac 3 2 q$, hence
 \[
  \sum_{i\geq q+2m-2} (i-(q+2))(i-(q+m)) n_i > \frac{2(m-2)^2}{\frac 32 q} q^2(q+2) = \frac 43 (m-2)^2 q(q+2).
 \]
 Recall that the sum in the left hand side of (\ref{Standard Eq:1}) has only non-negative terms, hence we obtain from Lemma \ref{Lm:Standard Eq}
 \[
  \frac 43 (m-2)^2 q < \frac1{q+2} \sum_{i>0} (i-(q+2))(i-(q+m)) n_i \leq qm^2 - (2q-1)m-q \left( \frac 52 q - 1 \right).
 \]
 We can rewrite this as
\[
 \frac 13 (m-5) qm + m + q \left( \frac 52 q - \frac {19} 3 \right) < 0.
\]
But since $q \geq 16$, we have $m > \sqrt{2 \cdot 16}$, hence $m \geq 8$, and all terms on the left hand side are positive, so we obtain a contradiction.
\end{proof}


\begin{thebibliography}{99}
\setlength{\itemsep}{0mm}

\bibitem[AD24]{AdriaensenDenaux2023}
S.~Adriaensen and L.~Denaux.
\newblock Small weight codewords of projective geometric codes {II}.
\newblock {\em Des. Codes Cryptogr.}, 92(9):2451--2472, 2024.

\bibitem[Adr23]{A3-2023-ANoteOn}
S.~Adriaensen.
\newblock A note on small weight codewords of projective geometric codes and on the smallest sets of even type.
\newblock {\em SIAM J. Discrete Math.}, 37:2072--2082, 2023.

\bibitem[AK92]{assmuskey1992}
E.~F. Assmus, Jr. and J.~D. Key.
\newblock {\em Designs and their codes}, volume 103 of {\em Cambridge Tracts in Mathematics}.
\newblock Cambridge University Press, Cambridge, 1992.

\bibitem[BI02]{bagchiinamdar}
B.~Bagchi and S.~P. Inamdar.
\newblock Projective geometric codes.
\newblock {\em J. Combin. Theory Ser. A}, 99(1):128--142, 2002.

\bibitem[CHKW03]{ClarkHatfieldKeyWard}
K.~L. Clark, L.~D. Hatfield, J.~D. Key, and H.~N. Ward.
\newblock Dual codes of projective planes of order 25.
\newblock {\em Advances in Geometry}, 2003(s1):140--152, 2003.

\bibitem[CKdR99]{calkinkeyderesmini}
N.~J. Calkin, J.~D. Key, and M.~J. de~Resmini.
\newblock Minimum weight and dimension formulas for some geometric codes.
\newblock {\em Des. Codes Cryptogr.}, 17(1-3):105--120, 1999.

\bibitem[CLMT25]{csajbok2025lower}
B.~Csajb{\'o}k, G.~Longobardi, G.~Marino, and R.~Trombetti.
\newblock A lower bound on the minimum weight of some geometric codes.
\newblock {\em arXiv preprint arXiv:2510.04307}, 2025.

\bibitem[DBVdV23]{deboeckvandevoorde2022}
M.~De~Boeck and G.~Van~de Voorde.
\newblock Embedded antipodal planes and the minimum weight of the dual code of points and lines in projective planes of order {$p^2$}.
\newblock {\em Des. Codes Cryptogr.}, 91(3):895--920, 2023.

\bibitem[Del70]{delsarte70bch}
Ph. Delsarte.
\newblock {${\rm BCH}$} bounds for a class of cyclic codes.
\newblock {\em SIAM J. Appl. Math.}, 19:420--429, 1970.

\bibitem[DGM70]{delsartegoethalsmacwilliams}
Ph. Delsarte, J.-M. Goethals, and F.~J. MacWilliams.
\newblock On generalized {R}eed-{M}uller codes and their relatives.
\newblock {\em Information and Control}, 16:403--442, 1970.

\bibitem[DWSV12]{dewinterschillewaertverstraete}
S.~De~Winter, J.~Schillewaert, and J.~Verstraete.
\newblock Large incidence-free sets in geometries.
\newblock {\em Electron. J. Combin.}, 19(4):Paper 24, 16, 2012.

\bibitem[Hae95]{haemers95}
W.~H. Haemers.
\newblock Interlacing eigenvalues and graphs.
\newblock {\em Linear Algebra Appl.}, 226/228:593--616, 1995.

\bibitem[KdR01]{KeydeResmini}
J.~D. Key and M.~J. de~Resmini.
\newblock Ternary dual codes of the planes of order nine.
\newblock {\em Journal of statistical planning and inference}, 95(1-2):229--236, 2001.

\bibitem[KM90]{KM}
G.~Korchm\'aros and F.~Mazzocca.
\newblock On {$(q+t)$}-arcs of type {$(0,2,t)$} in a {D}esarguesian plane of order {$q$}.
\newblock {\em Math. Proc. Cambridge Philos. Soc.}, 108(3):445--459, 1990.

\bibitem[KS20]{KissSzonyi}
G.~Kiss and T.~Sz{\H o}nyi.
\newblock {\em Finite geometries}.
\newblock CRC Press, Boca Raton, FL, [2020] \copyright 2020.

\bibitem[LSVdV08]{lavrauwstormevandevoorde2008kspaces}
M.~Lavrauw, L.~Storme, and G.~Van~de Voorde.
\newblock On the code generated by the incidence matrix of points and {\(k\)}-spaces in {$\mathrm{PG}(n,q)$} and its dual.
\newblock {\em Finite Fields Appl.}, 14(4):1020--1038, 2008.

\bibitem[LSVdV10]{LavrauwStormeVandeVoorde2010}
M.~Lavrauw, L.~Storme, and G.~Van~de Voorde.
\newblock Linear codes from projective spaces.
\newblock In {\em Error-correcting codes, finite geometries and cryptography}, volume 523 of {\em Contemp. Math.}, pages 185--202. Amer. Math. Soc., Providence, RI, 2010.

\bibitem[PZ18]{polverinozullo}
O.~Polverino and F.~Zullo.
\newblock Codes arising from incidence matrices of points and hyperplanes in {${\rm PG}(n,q)$}.
\newblock {\em J. Combin. Theory Ser. A}, 158:1--11, 2018.

\bibitem[SW18]{SzonyiWeiner}
T.~Sz{\H{o}}nyi and Zs. Weiner.
\newblock Stability of {\(k\bmod p\)} multisets and small weight codewords of the code generated by the lines of {PG}(2, {\(q\)}).
\newblock {\em J. Combin. Theory Ser. A}, 157:321--333, 2018.

\bibitem[Van19]{Vandendriessche}
Peter Vandendriessche.
\newblock Classification of the hyperovals in {${\rm PG}(2,64)$}.
\newblock {\em Electron. J. Combin.}, 26(2):Paper No. 2.35, 12, 2019.

\end{thebibliography}
\end{document}